\documentclass[11pt,oneside,english]{amsart}
\usepackage[T1]{fontenc}
\usepackage[latin9]{inputenc}
\usepackage[a4paper]{geometry}
\usepackage{amsthm}
\usepackage{amstext}
\usepackage{amssymb}

\makeatletter
\numberwithin{equation}{section} 
\numberwithin{figure}{section} 
  \theoremstyle{plain}
  \newtheorem*{thm*}{Theorem}
\theoremstyle{plain}
\newtheorem{thm}{Theorem}
  \theoremstyle{plain}
  \newtheorem{lem}[thm]{Lemma}
  \theoremstyle{plain}
  \newtheorem{cor}[thm]{Corollary}

\usepackage{xypic}
\theoremstyle{definition}

\newtheorem{enavant}[thm]{}

   \theoremstyle{remark}
  \newtheorem{rem}[thm]{Remark}

\makeatother

\usepackage{babel}

\begin{document}
\author{Adrien Dubouloz}  \address{Institut de Math\'ematiques de Bourgogne, Universit\'e de Bourgogne, 9 avenue Alain Savary - BP 47870, 21078 Dijon cedex, France}  \email{Adrien.Dubouloz@u-bourgogne.fr}
\author{Lucy Moser-Jauslin} \address{Institut de Math\'ematiques de Bourgogne, Universit\'e de Bourgogne, 9 avenue Alain Savary - BP 47870, 21078 Dijon cedex, France}
\email{moser@u-bourgogne.fr}
\author{Pierre-Marie Poloni} \address{Mathematisches Institut Universität Basel Rheinsprung 21 CH-4051 Basel Switzerland}
\email{pierre-marie.poloni@unibas.ch}

\title{Non cancellation for smooth contractible affine threefolds}
\begin{abstract}
We construct two non isomorphic contractible affine threefolds $X$
and $Y$ with the property that their cylinders $X\times\mathbb{A}^{1}$
and $Y\times\mathbb{A}^{1}$ are isomorphic, showing that the generalized
Cancellation Problem has a negative answer in general for contractible
affine threefolds. We also establish that $X$ and $Y$ are actually
biholomorphic as complex analytic varieties, providing the first example
of a pair of biholomorphic but not isomorphic exotic $\mathbb{A}^{3}$'s.
\end{abstract}
\maketitle

\section*{Introduction }

The Cancellation Problem asks if a complex algebraic variety $X$
of dimension $d$ such that $X\times\mathbb{A}^{n}$ is isomorphic
to $\mathbb{A}^{n+d}$ is isomorphic to $\mathbb{A}^{d}$. This is a
difficult problem in general and, apart form the trivial case
$d=1$, an affirmative answer is known only in dimension $2$. One
can ask more generally if two algebraic varieties $X$ and $Y$ such
that $X\times\mathbb{A}^{n}$ is isomorphic to
$Y\times\mathbb{A}^{n}$ for some $n\geq1$ are isomorphic. This
more general problem has an affirmative answer for a large class
of varieties: intuitively, cancellation should hold provided that
either $X$ or $Y$ does not contain too many rational curves. A
precise characterization has been given by Iitaka and Fujita
\cite{IiF77} in terms of logarithmic Kodaira dimension, namely, if
either $X$ or $Y$, say $X$, has non negative logarithmic Kodaira
dimension $\bar{\kappa}\left(X\right)\geq0$, then every
isomorphism between $X\times\mathbb{A}^{n}$ and
$Y\times\mathbb{A}^{n}$ descends to an isomorphism between $X$ and
$Y$. This assumption on the logarithmic Kodaira dimension turns
out to be essential. Indeed, W. Danielewski \cite{Dan89} showed in
1989 that the rational affine surfaces $S_{1}=\left\{
xy=z^{2}-1\right\} $ and $S_{2}=\left\{ x^{2}y=z^{2}-1\right\} $
in $\mathbb{A}^{3}$ are non isomorphic but have isomorphic
cylinders $S_{1}\times\mathbb{A}^{1}$ and
$S_{2}\times\mathbb{A}^{1}$. Since then, Danielewski's
construction has been generalized and adapted to construct many
new counter examples of the same type, in arbitrary dimension
\cite{Dub07,Fie94,Wil98}.

However, all counter-examples constructed so far using variants of
Danielewski's idea are remote from affine spaces: for instance,
the Danielewski surfaces have nontrivial Picard groups and their
underlying euclidean topological spaces are not contractible.
Therefore, one may expect that cancellation holds for affine
varieties close to the affine space. This is actually the case for
smooth contractible or factorial surfaces. For the first ones,
this follows from an algebro-geometric characterization of
$\mathbb{A}^{2}$ due to Miyanishi-Sugie \cite{Mi75,MiS80,Su89}
which says that a smooth acyclic surface $S$ with
$\bar{\kappa}\left(S\right)=-\infty$ is isomorphic to
$\mathbb{A}^{2}$ (see also \cite{CML08} for a purely algebraic
self-contained proof). On the other hand, the fact that
generalized cancellation holds for smooth factorial affine
surfaces $S$ seems to be folklore. Roughly, the argument goes as
follows: first one may assume that $S$ has logarithmic Kodaira
dimension $\bar{\kappa}=-\infty$. By virtue of a characterization
due to T. Sugie \cite{Su79}, it follows that $S$ admits an
$\mathbb{A}^{1}$-fibration $\pi:S\rightarrow C$ over a smooth
curve $C$, that is, a surjective morphism with general fibers
isomorphic to $\mathbb{A}^{1}$. The hypothesis that the Picard
group of $S$ is trivial implies that the same holds for $C$, and
so, $C$ is a factorial affine curve. Combined with the
classification of germs of degenerate fibers of
$\mathbb{A}^{1}$-fibrations given by K.-H. Fieseler \cite{Fie94},
the factoriality of $S$ implies that $\pi:S\rightarrow C$ has no
degenerate fibers, whence is a locally trivial
$\mathbb{A}^{1}$-bundle. Since $C$ is affine and factorial,
$\pi:S\rightarrow C$ is actually a trivial $\mathbb{A}^{1}$-bundle
$S\simeq C\times\mathbb{A}^{1}\rightarrow C$ and so, the result
follows from the affirmative answer to the generalized
Cancellation Problem for curves due to Abhyankar-Eakin-Heinzer
\cite{AbEaHe72}.

The situation turns out to be very different in dimension $3$.
Indeed, recently, D. Finston and S. Maubach \cite{FM08}
constructed smooth factorial counter-examples to the generalized
Cancellation Problem. The latter arise as  total spaces of
locally trivial $\mathbb{A}^{1}$-bundles over the complement of
the isolated singularity of a Brieskorn surface
$x^{p}+y^{q}+z^{r}=0$ in $\mathbb{A}^{3}$, with $1/p+1/q+1/r<1$.
By construction, these counter-examples are not
contractible, having the homology type of a $3$-sphere, 
and so, the existence of contractible counter-examples
remained open. A famous candidate for being such a counter-example
is the Russell cubic threefold $V$ defined by the equation
$x^{2}y+z^{2}+t^{3}+x=0$ in $\mathbb{A}^{4}$. The latter is known
to be contractible but not isomorphic to $\mathbb{A}^{3}$ (see
e.g. \cite{KaZa99} and \cite{ML96}) and it is an open problem
to decide whether $V\times\mathbb{A}^{1}$ is isomorphic to
$\mathbb{A}^{4}$ or not. In this article, we show that a mild
variation on the above candidate already leads to contractible
counter-examples to the generalized Cancellation Problem in
dimension $3$. Namely we consider the smooth affine threefolds \[
X_{a}=\left\{ x^{4}y+z^{2}+t^{3}+x+x^{2}+ax^{3}=0\right\} \] in
$\mathbb{A}^{4}$, where $a$ is a complex parameter. We establish
the following result:
\begin{thm*}
The threefolds $X_{a}$ are contractible, non isomorphic to $\mathbb{A}^{3}$
and not isomorphic to each other. However, the cylinders $X_{a}\times\mathbb{A}^{1}$,
$a\in\mathbb{C}$, are all isomorphic.
\end{thm*}
Recall that by virtue of a characterization due to A. Dimca \cite{Di90},
the varieties $X_{a}$, $a\in\mathbb{C}$, are all diffeomorphic to
$\mathbb{R}^{6}$ when equipped with the euclidean topology, whence
give examples of non isomorphic exotic affine spaces. We show in contrast
they are all biholomorphic when considered as complex analytic manifolds,
thus answering an open problem raised by M. Zaidenberg \cite{Zai05}.
\\

The article is organized as follows. In the first section, we
consider more general contractible affine threefolds $X_{n,p}$ in
$\mathbb{A}^{4}$ defined by equations of the form
$x^{n}y+z^{2}+t^{3}+xp\left(x\right)=0$, where $n\geq2$ and
$p\left(x\right)\in\mathbb{C}\left[x\right]$. We provide, for each
fixed integer $n\geq2$, a complete classification of isomorphism
classes of such varieties and their cylinders. As a corollary, we
obtain that the varieties $X_{a}$, $a\in\mathbb{C}$, are pairwise
non isomorphic and have isomorphic cylinders. The second section
is devoted to a geometric interpretation of the existence of an
isomorphism between the cylinders $X_{a}\times\mathbb{A}^{1}$,
$a\in\mathbb{C}$, in terms of a Danielewski fiber product trick
construction.

\section{Main results }

For any integer $n\geq2$ and any polynomial
$q\in\mathbb{C}\left[x,z,t\right]$, we consider the affine
threefold $V_{n,q}$ in $\mathbb{A}^{4}={\rm
Spec}\left(\mathbb{C}\left[x,y,z,t\right]\right)$ defined by the
equation \[ x^{n}y+z^{2}+t^{3}+xq\left(x,z,t\right)=0.\] Note that
$V_{n,q}$ is smooth if and only if $q\left(0,0,0\right)$ is a
nonzero constant. The morphism $\pi={\rm
pr}_{x}:V_{n,q}\rightarrow\mathbb{A}^{1}$ is a flat
$\mathbb{A}^{2}$-fibration restricting to a trivial
$\mathbb{A}^{2}$-bundle over $\mathbb{A}^{1}\setminus\left\{
0\right\} $ and with degenerate fiber $\pi^{-1}\left(0\right)$
isomorphic to the cylinder $\Gamma_{2,3}\times\mathbb{A}^{1}$ over
the plane cuspidal curve $\Gamma_{2,3}=\left\{
z^{2}+t^{3}=0\right\} $. This implies in particular that $V_{n,q}$
is factorial. Moreover, via the natural localization homomorphism,
one may identify the coordinate ring of $V_{n,q}$ with the
sub-algebra
$\mathbb{C}\left[x,z,t,x^{-n}\left(z^{2}+t^{3}+xq\left(x,y,t\right)\right)\right]$
of $\mathbb{C}\left[x^{\pm1},z,t\right]$. This says equivalently
that $V_{n,q}$ is the \emph{affine modification}
$\sigma_{n,q}={\rm pr}_{x,z,t}:V_{n,q}\rightarrow\mathbb{A}^{3}$
of $\mathbb{A}^{3}={\rm
Spec}\left(\mathbb{C}\left[x,z,t\right]\right)$ with center at the
closed subscheme $Z_{n,q}$ with defining ideal
$I_{n,q}=\left(x^{n},z^{2}+t^{3}+xq\left(x,y,t\right)\right)$ and
divisor $D=\left\{ x^{n}=0\right\} $ in the sense of
\cite{KaZa99}, that is, $V_{n,q}$ is isomorphic to the complement
of the proper transform of $D$ in the blow-up of $\mathbb{A}^{3}$
with center at $Z_{n,q}$. It follows in particular from Theorem
3.1 in \cite{KaZa99} that a smooth $V_{n,q}$ is contractible when
considered as a complex manifold. As a consequence of the general
methods developed in \cite{K-ML07}, we have the following useful
criterion to decide which threefolds $V_{n,q}$ are isomorphic.
\begin{lem}
\label{lem:Iso} Every isomorphism $\Phi:V_{n_{1},q_{1}}\stackrel{\sim}{\longrightarrow}V_{n_{2},q_{2}}$
is the lift via $\sigma_{n_{1},q_{1}}$ and $\sigma_{n_{2},q_{2}}$
of an automorphism of $\mathbb{A}^{3}$ which maps the locus of the
modification $\sigma_{n_{1},q_{1}}$ isomorphically onto the one of
the modification $\sigma_{n_{2},q_{2}}$. Equivalently, there exists
a commutative diagram \[\xymatrix{ V_{n_1,q_1} \ar[d]_{\sigma_{n_1,q_1}} \ar[r]^{\Phi} & V_{n_2,q_2} \ar[d]^{\sigma_{n_2,q_2}} \\ \mathbb{A}^3 \ar[r]^{\varphi} & \mathbb{A}^3 }\]
where $\varphi$ is an automorphism of $\mathbb{A}^{3}$ which preserves
the hyperplane $\left\{ x=0\right\} $ and maps $Z_{n_{1},q_{1}}$
isomorphically onto $Z_{n_{2},q_{2}}$. \end{lem}
\begin{proof}
The fact that every automorphism of $\mathbb{A}^{3}$ satisfying
the above property lifts to an isomorphism between
$V_{n_{1},q_{1}}$ and $V_{n_{2},q_{2}}$ is an immediate
consequence of the universal property of affine modifications, Proposition 2.1 in
\cite{KaZa99}. For the converse we exploit two invariants of an
affine variety $V$: the Makar-Limanov invariant (resp. the Derksen
invariant) of $V$ which is the sub-algebra ${\rm
ML}\left(V\right)$ (resp. ${\rm Dk}\left(V\right)$) of
$\Gamma\left(V,\mathcal{O}_{V}\right)$ generated by regular
functions invariant under all (resp. at least one) non trivial
algebraic $\mathbb{G}_{a}$-actions on $V$ (see e.g. \cite{Zai99}).
The same arguments as the ones used to treat the case of the
Russell cubic threefold $V_{2,1}$ in \cite{K-ML07} show more
generally that ${\rm
ML}\left(V_{n,q}\right)=\mathbb{C}\left[x\right]$ and ${\rm
Dk}\left(V_{n,q}\right)=\mathbb{C}\left[x,z,t\right]$ for every
$q\in\mathbb{C}\left[x,z,t\right]$. This implies that any
isomorphism between the coordinate rings of $V_{n_{1},q_{1}}$ and
$V_{n_{2},q_{2}}$ restricts to an automorphism $\varphi^{*}$ of
$\mathbb{C}\left[x,z,t\right]$ inducing a one $x\mapsto ax+b$ of
$\mathbb{C}\left[x\right]$, where $a\in\mathbb{C}^{*}$ and
$b\in\mathbb{C}$. Actually, $b=0$ as the zero set of $ax+b$ in
$V_{n_{i},q_{i}}$,$i=1,2$, is singular if and only if $b=0$. So
$\varphi^{*}$ stabilizes the ideal $\left(x\right)$. In turn, the
fact that 
$I_{n_{i},q_{i}}=x^{n}\Gamma\left(V_{n_{i},q_{i}},\mathcal{O}_{V_{n_{i},q_{i}}}\right)\cap\mathbb{C}\left[x,z,t\right]$
implies that
$\varphi^{*}\left(I_{n_{2},q_{2}}\right)=I_{n_{1},q_{1}}$. Now the
assertion follows since the modification morphism
$\sigma_{n,q}:V_{n,q}\rightarrow\mathbb{A}^{3}$ defined above is
precisely induced by the natural inclusion of ${\rm
Dk}\left(V_{n,q}\right)$ into 
$\Gamma\left(V_{n,q},\mathcal{O}_{V_{n,q}}\right)$.
\end{proof}
\begin{enavant} From now on, we only consider a very special case
of smooth contractible threefolds $V_{n,q}$, namely, the ones $V_{n,p}$
defined by equations \[
x^{n}y+z^{2}+t^{3}+xp\left(x\right)=0,\]
where $p\in\mathbb{C}\left[x\right]$ is a polynomial such that $p\left(0\right)\neq0$.
We have the following result.
\end{enavant}

\begin{thm}\label{main-thm}
For a fixed integer $n\geq2$, the following hold:
\begin{enumerate}
\item The algebraic varieties $V_{n,p_{1}}$ and $V_{n,p_{2}}$ are
isomorphic if and only if there exists
$\lambda,\varepsilon\in\mathbb{C}^{*}$ such that
$p_{2}\left(x\right)\equiv\varepsilon p_{1}\left(\lambda
x\right)\;\textrm{mod }x^{n-1}$.

\item The cylinders $V_{n,p}\times\mathbb{A}^{1}$ are all isomorphic.

\item The varieties $V_{n,p}$ are all isomorphic as complex analytic
manifolds.\end{enumerate} \end{thm}

\begin{proof}
Letting $r=z^{2}+t^{3}$, it follows from Lemma \ref{lem:Iso} above
that $V_{n,p_{2}}\simeq V_{n,p_{1}}$ if and only if there exists
an automorphism $\phi$ of $\mathbb{A}^{3}={\rm Spec}\left(\mathbb{C}\left[x,z,t\right]\right)$
which preserves the hyperplane $\left\{ x=0\right\} $ and maps the
closed subscheme with defining ideal $\left(x^{n},r+xp_{2}\left(x\right)\right)$
isomorphically onto the one with defining ideal $\left(x^{n},r+xp_{1}\left(x\right)\right)$.
 Since such an automorphism stabilizes the hyperplane $\left\{ x=0\right\} $,
there exists $\lambda\in\mathbb{C}^{*}$ such that $\phi^{*}\left(x\right)=\lambda x$. Furthermore,
$\phi$ maps the curve $\Gamma_{2,3}=\left\{ x=z^{2}+t^{3}=0\right\} $
isomorphically onto itself.

So there exists $\mu\in\mathbb{C}^{*}$ such that
$\phi^{*}z\equiv\mu^{3}z\;\textrm{mod }x$ and
$\phi^{*}t\equiv\mu^{2}t\;\textrm{mod }x$. Therefore, by composing
$\phi$  with the automorphism
$\theta:\mathbb{A}^{3}\stackrel{\sim}{\rightarrow}\mathbb{A}^{3}$,
$\left(x,z,t\right)\mapsto\left(\lambda^{-1}x,\mu^{-3}z,\mu^{-2}t\right)$,
we get an automorphism $\psi$ of $\mathbb{A}^{3}$ such that
$\psi^{*}x=x$, $\psi^{*}z\equiv z\;\textrm{mod }x$,
$\psi^{*}t\equiv t\;\textrm{mod }x$ and which maps the closed
subscheme with defining ideal
$\left(x^{n},r+xp_{2}\left(x\right)\right)$ isomorphically onto
the one with defining ideal $\left(x^{n},r+\mu^{6}\lambda
xp_{1}\left(\lambda x\right)\right)$. Letting
$\varepsilon=\mu^{6}\lambda$, this implies
$p_{2}\left(x\right)\equiv\varepsilon p_{1}\left(\lambda
x\right)\;\textrm{mod }x^{n-1}$ by virtue of Lemma
\ref{lem:Equ-Crit} below.

Conversely, if $p_{2}\left(x\right)\equiv\varepsilon
p_{1}\left(\lambda x\right)\;\textrm{mod }x^{n-1}$, we let
$\mu\in\mathbb{C}^*$ be such that $\varepsilon=\mu^{6}\lambda$.
Then the automorphism
$$\left(x,y,z,t\right)\mapsto\left(\lambda
x,\lambda^{-n}\mu^{-6}y+\lambda^{-n}x^{-n+1}\left(\mu^{-6}p_2(x)-\lambda
p_{1}\left(\lambda x\right)\right),\mu^{-3}z,\mu^{-2}t\right)$$of
$\mathbb{A}^{4}$ maps $V_{n,p_{2}}$ isomorphically onto
$V_{n,p_{1}}$. This proves (1).

To prove (2) and (3), it is enough to show that for every $p\in\mathbb{C}\left[x\right]$
such that $p\left(0\right)\neq0$, $V_{n,p}$ is biholomorphic to
$V_{n,1}$ and that these two threefolds have algebraically isomorphic
cylinders. The arguments are similar to arguments developed in \cite{MJ-P06}.

First, up to the composition by an isomorphism induced by an automorphism
of $\mathbb{A}^{4}$ of the form $\left(x,y,z,t\right)\mapsto\left(\lambda x,\lambda^{-n}y,z,t\right)$,
 we may assume that $p\left(0\right)=1$. Remark also that the ideals $(x^n,z^2+t^3+x)$ and $(x^n,p(x)(z^2+t^3+x))$ are equal. Therefore, by virtue of Lemma \ref{lem:Iso}, $V_{n,1}$ is isomorphic as an
algebraic variety to the variety $W_{n,p}$ defined by the equation $x^ny+p(x)(z^2+t^3+x)=0$.

Letting $f\in\mathbb{C}\left[x\right]$
be a polynomial such that $\exp\left(xf\left(x\right)\right)\equiv p\left(x\right)\;\textrm{mod }x^{n}$,
one checks that the biholomorphism $\psi$ of $\mathbb{A}^{3}$ defined
by \[
\Psi\left(x,z,t\right)=\left(x,y-\frac{\exp\left(xf\left(x\right)\right)-p(x)}{x^n}(z^2+t^3),\exp\left(\frac{1}{2}xf\left(x\right)\right)z,\exp\left(\frac{1}{3}xf\left(x\right)\right)t\right)\]
maps  $W_{n,p}$ onto $V_{n,p}$. So (3) follows.

For (2), we choose polynomials $g_{1}\in\mathbb{C}\left[x\right]$
such that $\exp\left(\frac{1}{2}xf\left(x\right)\right)\equiv
g_{1}\;\textrm{mod }x^{n}$ and $g_{2}\in\mathbb{C}\left[x\right]$
relatively prime to $g_{1}$ such that
$\exp\left(\frac{1}{3}xf\left(x\right)\right)\equiv
g_{2}\;\textrm{mod }x^{n}$. Since
$g_{1}\left(0\right)=g_{2}\left(0\right)=1$, the polynomials
$x^{n}g_{1}$, $x^{n}g_{2}$ and $g_{1}g_{2}$ generate the unit
ideal in $\mathbb{C}\left[x\right]$. So there exist polynomials
$h_{1},h_{2},h_{3}\in\mathbb{C}\left[x\right]$ such that
\[
\left(\begin{array}{ccc}
g_{1} & 0 & x^n\\
0 & g_{2} & x^n\\
h_1(x) & h_2(x) & h_{3}(x)\end{array}\right)\in{\rm GL}_{3}\left(\mathbb{C}\left[x\right]\right).\]

This matrix defines a $\mathbb{C}\left[x\right]$-automorphism of
$\mathbb{C}\left[x\right]\left[z,t,w\right]$ which maps the ideal
$\left(x^{n},r+xp\left(x\right)\right)$ of
$\mathbb{C}\left[x\right]\left[z,t,w\right]$ onto the one
$\left(x^{n},p(x)(r+x)\right)=\left(x^{n},r+x\right)$. Since these
ideals coincide with the centers of the affine modifications
$\sigma_{n,p}\times{\rm
id}:V_{n,p}\times\mathbb{A}^{1}\rightarrow\mathbb{A}^{4}$ and
$\sigma_{n,1}\times{\rm
id}:V_{n,1}\times\mathbb{A}^{1}\rightarrow\mathbb{A}^{4}$
respectively, we can conclude by Proposition 2.1 in \cite{KaZa99}
that the corresponding automorphism of $\mathbb{A}^{4}={\rm
Spec}\left(\mathbb{C}\left[x,z,t,w\right]\right)$ lifts to an
isomorphism between $V_{n,1}\times\mathbb{A}^{1}$ and
$V_{n,p}\times\mathbb{A}^{1}$. This completes the proof.
\end{proof}

\begin{lem}
\label{lem:Equ-Crit} Let $n\geq2$ and $p_{1},p_{2}\in\mathbb{C}\left[x\right]$
be polynomials of degree $\leq n-2$. If there exists a $\mathbb{C}\left[x\right]$-automorphism
$\Phi$ of $\mathbb{C}\left[x\right]\left[z,t\right]$ such that $\Phi\equiv{\rm id}\;\textrm{mod }x$
and $\Phi\left(x^{n},z^{2}+t^{3}+xp_{1}\right)=\left(x^{n},z^{2}+t^{3}+xp_{2}\right)$
then $p_{1}=p_{2}$. \end{lem}
\begin{proof}
We let $r=z^{2}+t^{3}$ and we let
$p_{i}=\sum_{k=0}^{n-2}a_{ik}x^{k}$, $i=1,2$. We let $n_{0}\geq1$
be the largest integer such that $\Phi\equiv{\rm id}\;\textrm{mod
}x^{n_{0}}$. If $n_{0}\geq n-1$ then we are done. Otherwise, there
exist $\alpha,\beta\in\mathbb{C}\left[x,z,t\right]$ such that
$\Phi\left(r+xp_{1}\right)=\left(1+x^{n_{0}}\alpha\right)\left(r+xp_{2}\right)+x^{n}\beta$.
Since the determinant of the Jacobian of $\Phi$ is a nonzero
constant, there exists $h\in\mathbb{C}\left[z,t\right]$ such that
\[
\begin{cases}
\Phi\left(z\right) & \equiv z+x^{n_{0}}\partial_{t}h\;\textrm{mod }x^{n_{0}+1}\\
\Phi\left(t\right) & \equiv t-x^{n_{0}}\partial_{z}h\;\textrm{mod
}x^{n_{0}+1}.\end{cases}\] It follows that
$\Phi\left(r+xp_{1}\right)\equiv r+xp_{1}+x^{n_{0}}{\rm
Jac}\left(r,h\right)\;\textrm{mod }x^{n_{0}+1}$. By comparing with
the other expression for $\Phi\left(r+xp_{1}\right)$, we find that
$p_{1}\equiv p_{2}\;\textrm{mod }x^{n_{0}-1}$ and
$a_{1,n_{0}-1}+{\rm
Jac}\left(r,h\right)=\alpha\left(0,z,t\right)r+a_{2,n_{0}-1}$.
Since ${\rm
Jac}\left(r,h\right)\in\left(z,t\right)\mathbb{C}\left[z,t\right]$,
we obtain $a_{1,n_{0}-1}=a_{2,n_{0}-1}$ and ${\rm
Jac}\left(r,h\right)=\alpha\left(0,z,t\right)r$. Moreover, the
condition ${\rm Jac}\left(r,h\right)\in
r\mathbb{C}\left[z,t\right]$ implies that $h=\gamma(z,t)r+c$ for
some $\gamma\in\mathbb{C}\left[z,t\right]$ and $c\in\mathbb{C}$.
Now we consider the exponential
$\mathbb{C}\left[x\right]/\left(x^{n}\right)$-automorphism
$\exp\left(\delta\right)$ of
$\mathbb{C}\left[x\right]/\left(x^{n}\right)\left[z,t\right]$
associated with the Jacobian derivation \[ \delta=x^{n_{0}}{\rm
Jac}\left(\cdot,\gamma\left(z,t\right)\left(r+xp_{1}\right)\right).\]
Since the determinant of the Jacobian of $\exp\left(\delta\right)$
is equal to $1$ (see \cite{MJ09}), it follows from \cite{vE-M-V}
that there exists a $\mathbb{C}\left[x\right]$-automorphism
$\Theta$ of $\mathbb{C}\left[x\right]\left[z,t\right]$ such that
$\Theta\equiv\exp\left(\delta\right)\;\textrm{mod }x^{n}$. By
construction, $\Theta\equiv\Phi\;\textrm{mod }x^{n_{0}+1}$ and, 
since $r+xp_{1}\in{\rm Ker}\delta$, $\Theta$ preserves the ideal
$\left(x^{n},r+xp_{1}\right)$. It follows that
$\Psi=\Phi\circ\Theta^{-1}$ is a
$\mathbb{C}\left[x\right]$-automorphism of
$\mathbb{C}\left[x\right]\left[z,t\right]$ such that
$\Psi\left(x^{n},z^{2}+t^{3}+xp_{1}\right)=\left(x^{n},z^{2}+t^{3}+xp_{2}\right)$
and such that $\Psi\equiv{\rm id}\;\textrm{mod }x^{n_{0}+1}$. Now
the assertion follows by induction.
\end{proof}

\begin{rem}
  In the proofs above, the crucial point is to characterize the existence of isomorphisms between the centers  $Z_{n,p}$ of the affine modifications defining the threefolds $V_{n,p}$ that are induced by automorphisms of the ambient space $\mathbb{A}^3$. Note that for fixed integer $n\geq2$, these closed subschemes $Z_{n,p}$ with defining ideals $I_{n,p}=\left(x^{n},z^{2}+t^{3}+xp(x)\right)$ are all isomorphic as abstract schemes, and even as abstract infinitesimal deformations of the plane cubic $\Gamma_{2,3}=\left\{ z^{2}+t^{3}=0\right\}$ over ${\rm Spec}\left(\mathbb{C}\left[x\right]/\left(x^{n}\right)\right)$. Indeed, letting $g_1(x),\,g_2(x)\in \mathbb{C}\left[x\right]$ be polynomials such that $(g_1(x))^2\equiv p(x)\;\textrm{mod }x^n$ and $(g_2(x))^3\equiv p(x)\;\textrm{mod }x^n$, one checks for instance
that the automorphism $\xi$ of ${\rm Spec}\left(\mathbb{C}\left[x\right]/\left(x^{n}\right)\left[z,t\right]\right)$
defined by \[\xi\left(x,z,t\right)=\left(x,g_1(x)z,g_2(x)t\right)\]
induces an isomorphism between $Z_{n,p}$ and $Z_{n,1}$. However, Theorem \ref{main-thm} says in particular that no
isomorphism of this kind can be lifted to an automorphism of $\mathbb{A}^{3}={\rm Spec}\left(\mathbb{C}\left[x,z,t\right]\right)$. In other words, the $Z_{n,p}$'s can be considered as defining non-equivalent closed embeddings of $Z_{n,1}$ in $\mathbb{A}^3$. 
 \indent The above proof also gives counterexamples, in dimension four, to the
so-called stable equivalence problem (see \cite{MLRSY} and
\cite{MJ-P06}). Indeed, it implies that for each $n\geq2$ and
each $p(x)\in\mathbb{C}[x]$ with $p(0)\neq0$, the polynomials
$x^ny+z^2+t^3+xp(x)$ and $x^ny+p(x)(z^2+t^3+x)$ are equivalent by
an automorphism of $\mathbb{C}[x,y,z,t,w]$ whereas they are not
equivalent up to an automorphism of $\mathbb{C}[x,y,z,t]$. Their
zero-sets are even non isomorphic smooth affine threefolds. 
\end{rem}

As a very particular case of the above discussion, we obtain the result
announced in the introduction, namely:
\begin{cor}
The smooth contractible affine threefolds $X_{a}=\left\{ x^{4}y+z^{2}+t^{3}+x+x^{2}+ax^{3}=0\right\} $,
$a\in\mathbb{C}$, are pairwise non isomorphic. However, their cylinders
$X_{a}\times\mathbb{A}^{1}$, $a\in\mathbb{C}$, are all isomorphic.
\end{cor}

\section{A geometric interpretation }

Here we give a geometric interpretation of the existence of an isomorphism
between the cylinders over the varieties $X_{a}$, $a\in\mathbb{C}$,
in terms of a variant of the famous Danielewski fiber product trick
\cite{Dan89}. Of course, the construction below can be adapted to
cover the general case, but we find it more enlightening to only consider
the particular case of the varities $X_{0}$ and $X_{1}$ in $\mathbb{A}^{4}={\rm Spec}\left(\mathbb{C}\left[x,y,z,t\right]\right)$
defined respectively by the equations \[
x^{4}y+z^{2}+t^{3}+x+x^{2}=0\quad\textrm{and}\quad x^{4}y+z^{2}+t^{3}+x+x^{2}+x^{3}=0.\]
For our purpose, it is convenient to use the fact that $X_{0}$ and
$X_{1}$ are isomorphic to the varieties $X$ and $Y$ in $\mathbb{A}^{4}$
defined respectively by the equations \[
x^{4}z=y^{2}+x+x^{2}-t^{3}\quad\textrm{and}\quad x^{4}z=\left(1+\alpha x^{2}\right)y^{2}+x+x^{2}-t^{3},\]
where $\alpha=-\frac{5}{3}$ . Clearly, the first isomorphism is simply
induced by the coordinate change $\left(x,y,z,t\right)\mapsto\left(x,z,-y,-t\right)$.
For the second one, one checks first that for $\beta=-1/3$, the following matrix in
${\rm GL}_{2}\left(\mathbb{C}\left[x\right]\right)$ \[
\left(\begin{array}{cc}
1-\beta x^{2}+\frac{1}{2}\beta^{2}x^{4} & \frac{1}{2}\beta^{2}x^{4}\\
\frac{1}{2}\beta^{2}x^{4} & 1+\beta x^{2}+\frac{1}{2}\beta^{2}x^{4}\end{array}\right)\]
defines a $\mathbb{C}\left[x\right]$-automorphism of
 $\mathbb{C}\left[x\right]\left[z,t\right]$ which maps the ideal 
$\left(x^{4},\left(1+\alpha x^{2}\right)z^{2}+x+x^{2}+t^{3}\right)$
onto the one $\left(x^{4},z^{2}+x+x^{2}+x^{3}+t^{3}\right)$. By virtue
of Lemma \ref{lem:Iso}, the corresponding automorphism of $\mathbb{A}^{3}$
lifts to an isomorphism between $X_{1}$ and the subvariety of $\mathbb{A}^{4}$
defined by the equation $x^{4}y+\left(1+\alpha x^{2}\right)z^{2}+x+x^{2}+t^{3}=0$,
and so, we eventually get the desired isomorphism by composing with
the previous coordinate change.

\begin{enavant} Now, the principle is the following: we  observe
that both $X$ and $Y$ come equipped with  $\mathbb{G}_{a}$-actions
induced by the ones on $\mathbb{A}^{4}$ associated respectively with
the locally nilpotent derivations $x^{4}\partial_{y}+2y\partial_{z}$
and $x^{4}\partial_{y}+2\left(1+\alpha x^{2}\right)y\partial_{z}$ of $\mathbb{C}\left[x,y,z,t\right]$.
The latter restrict to free actions on the open subsets $X^{*}=X\setminus\left\{ x=t=0\right\} $
and $Y^{*}=Y\setminus\left\{ x=t=0\right\} $ of $X$ and $Y$ respectively,
and so, they admit quotients $X^{*}\rightarrow X^{*}/\mathbb{G}_{a}$
and $Y^{*}\rightarrow Y^{*}/\mathbb{G}_{a}$ in the form of \'etale
locally trivial $\mathbb{G}_{a}$-bundles over suitable algebraic
spaces. We first check that $X^{*}/\mathbb{G}_{a}$ and $Y^{*}/\mathbb{G}_{a}$
are isomorphic to a same algebraic space $\mathfrak{S}$. This implies
that the fiber product $W=X^{*}\times_{\mathfrak{S}}Y^{*}$ has the
structure of a locally trivial $\mathbb{G}_{a}$-bundle over both
$X^{*}$ and $Y^{*}$ via the first and the second projection respectively.
Since $X^{*}$ and $Y^{*}$ are both strictly quasi-affine, there
is no guarantee a priori that these $\mathbb{G}_{a}$-bundles are
trivial. But we check below that it is indeed the case. Therefore,
since $X$ and $Y$ are affine, normal, and $X\setminus X^{*}$ and $Y\setminus Y^{*}$
have codimension $2$ in $X$ and $Y$ respectively, the corresponding
isomorphism $X^{*}\times\mathbb{A}^{1}\simeq W\simeq Y^{*}\times\mathbb{A}^{1}$
extends to a one $X\times\mathbb{A}^{1}\simeq Y\times\mathbb{A}^{1}$
.

\end{enavant}

\begin{enavant} \label{Par-Quot-Space} Let us check first that the
quotient spaces $X^{*}/\mathbb{G}_{a}$ and $Y^{*}/\mathbb{G}_{a}$
are indeed isomorphic. The restriction of the projection ${\rm pr}_{x,t}:\mathbb{A}^{4}\rightarrow\mathbb{A}^{2}={\rm Spec}\left(\mathbb{C}\left[x,t\right]\right)$
to $X^{*}$ and $Y^{*}$ induces $\mathbb{G}_{a}$-invariant morphisms
$\alpha:X^{*}\rightarrow\mathbb{A}^{2}\setminus\left\{ \left(0,0\right)\right\} $
and $\beta:Y^{*}\rightarrow\mathbb{A}^{2}\setminus\left\{ 0,0\right\}$. The latter restrict to trivial $\mathbb{G}_{a}$-bundles over $\mathbb{A}^{2}\setminus\left\{ x=0\right\} $. In contrast, the fiber
of each morphism over a closed point $\left(0,t\right)\in\mathbb{A}^{2}\setminus\left\{ \left(0,0\right)\right\} $
consists of the disjoint union of two affine lines $\left\{ x=0,y=\pm\mu\right\} $
where $\mu$ is a square root of $t^3$ whereas the fiber over the
non closed point $\left(x\right)\in\mathbb{C}\left[x,t\right]$ with
residue field $\mathbb{C}\left(t\right)$ corresponding to the punctured
line $\left\{ x=0\right\} \subset\mathbb{A}^{2}\setminus\left\{ \left(0,0\right)\right\}$ is
isomorphic to the affine line over the degree $2$ Galois extension
$\mathbb{C}\left(t\right)\left[y\right]/\left(y^{2}-t^{3}\right)$
of $\mathbb{C}\left(t\right)$. This indicates that the quotient spaces
$X^{*}/\mathbb{G}_{a}$ and $Y^{*}/\mathbb{G}_{a}$ should be obtained
from $\mathbb{A}^{2}\setminus\left\{ \left(0,0\right)\right\} $ by
replacing the punctured line $\left\{ x=0\right\} $ by a nontrivial
double \'etale covering of itself. An algebraic space $\mathfrak{S}$
with this property can be constructed in two steps as follows : first
we let $U_{\lambda}={\rm Spec}\left(\mathbb{C}\left[x,\lambda^{\pm1}\right]\right)$,
$U_{\lambda\lambda}={\rm Spec}\left(\mathbb{C}\left[x^{\pm1},\lambda^{\pm1}\right]\right)$
and we define an algebraic space $\mathfrak{S}_{\lambda}$ as the
quotient of $U_{\lambda}$ by the following \'etale equivalence relation

\[
\left(s,t\right):R_{\lambda}=U_{\lambda}\sqcup U_{\lambda\lambda}\longrightarrow U_{\lambda}\times U_{\lambda},\;\begin{cases}
U_{\lambda}\ni\left(x,\lambda\right) & \mapsto\left(\left(x,\lambda\right),\left(x,\lambda\right)\right)\\
U_{\lambda\lambda}\ni\left(x,\lambda\right) & \mapsto\left(\left(x,\lambda\right),\left(x,-\lambda\right)\right).\end{cases}\]
By construction, the $R_{\lambda}$-invariant morphism $U_{\lambda}\rightarrow{\rm Spec}\left(\mathbb{C}\left[x,t^{\pm1}\right]\right)$,
$\left(x,\lambda\right)\mapsto\left(x,\lambda^{2}\right)$ descends
to a morphism $\mathfrak{S}_{\lambda}\rightarrow{\rm Spec}\left(\mathbb{C}\left[x,t^{\pm1}\right]\right)$
restricting to an isomorphism over ${\rm Spec}\left(\mathbb{C}\left[x^{\pm1},t^{\pm1}\right]\right)$.
The fiber over the punctured line $\left\{ x=0\right\} $ is isomorphic
to ${\rm Spec}\left(\mathbb{C}\left(t\right)\left[\lambda\right]/\left(\lambda^{2}-t^{3}\right)\right)$.
Now we let $\mathfrak{S}$ be the algebraic space obtained by gluing $\mathfrak{S}_{\lambda}$ and
$U_{x}={\rm Spec}\left(\mathbb{C}\left[x^{\pm1},t\right]\right)$
by the identity on ${\rm Spec}\left(\mathbb{C}\left[x^{\pm1},t^{\pm1}\right]\right)$.
By construction, $\mathfrak{S}$ comes equipped with an \'etale cover
$p:V\rightarrow\mathfrak{S}$ by the scheme $V=U_{x}\sqcup U_{\lambda}$.
We let $U_{x,\lambda}=U_{x}\times_{\mathfrak{S}}U_{\lambda}\simeq{\rm Spec}\left(\mathbb{C}\left[x^{\pm1},\lambda^{\pm1}\right]\right)$.

\end{enavant}
\begin{lem}
\label{lem:Iso-Quot-Spaces} The quotients spaces $X^{*}/\mathbb{G}_{a}$
and $Y^{*}/\mathbb{G}_{a}$ are both isomorphic to $\mathfrak{S}$. \end{lem}
\begin{proof}
The argument is very similar to the one used in \cite{Dub09}.

1) The case of $X^{*}$. This quasi-affine threefold is covered by
two $\mathbb{G}_{a}$-invariant open subsets \[
V_{x}=X^{*}\setminus\left\{ x=0\right\} =X\setminus\left\{ x=0\right\} \qquad\textrm{and}\qquad V_{t}=X^{*}\setminus\left\{ t=0\right\} =X\setminus\left\{ t=0\right\} \]
Letting $U_{x}={\rm Spec}\left(\mathbb{C}\left[x^{\pm1},t\right]\right)$,
one checks easily that the morphism \[
U_{x}\times\mathbb{G}_{a}\longrightarrow V_{x},\:\left(x,t,v\right)\mapsto\left(x,x^{4}v,x^{4}v^{2}+x^{-4}\left(-t^{3}+x+x^{2}\right),t\right)\]
is an isomorphism, equivariant for the $\mathbb{G}_{a}$-action on
$U_{x}\times\mathbb{G}_{a}$ by translations on the second factor,
which yields a trivialization of the induced $\mathbb{G}_{a}$-action
on $V_{x}$. In contrast, the induced action on $V_{t}$ is not trivial.
However, letting $U_{\lambda}={\rm Spec}\left(\mathbb{C}\left[x,\lambda^{\pm1}\right]\right),$
we claim that there exists $\sigma,\xi\in\mathbb{C}\left[x,\lambda^{\pm1}\right]$
such that the morphism \[
U_{\lambda}\times\mathbb{G}_{a}\longrightarrow V_{t},\;\left(x,\lambda,v\right)\mapsto\left(x,x^{4}v+\sigma,\left(x^{4}v+2\sigma\right)v+\xi,\lambda^{2}\right)\]
is \'etale and equivariant for the $\mathbb{G}_{a}$-action on $U_{\lambda}\times\mathbb{G}_{a}$
by translations on the second factor, whence defines an \'etale trivialization
on the induced action on $V_{t}$. This can be seen as follows : let
\[
V_{\lambda}=V_{t}\times_{\mathbb{A}_{*}^{1}}\mathbb{A}_{*}^{1}\simeq{\rm Spec}\left(\mathbb{C}\left[x,y,z,\lambda^{\pm1}\right]/\left(x^{4}z-y^{2}+\lambda^{6}-x-x^{2}\right)\right)\]
be the pull-back of $V_{t}$ by the Galois covering \[
\varphi:\mathbb{A}_{*}^{1}={\rm Spec}\left(\mathbb{C}\left[\lambda^{\pm1}\right]\right)\rightarrow\mathbb{A}_{*}^{1}={\rm Spec}\left(\mathbb{C}\left[t^{\pm1}\right]\right),\;\lambda\mapsto t=\lambda^{2}.\]
Since $\lambda\in\mathbb{C}\left[x,\lambda^{\pm1}\right]$ is invertible
it follows that one can find $\sigma\in\mathbb{C}\left[x,\lambda^{\pm1}\right]$
with $\deg_{x}\sigma\leq3$ and $\sigma\left(0,\lambda\right)=\lambda^{3}$,
and $\xi\in\mathbb{C}\left[x,\lambda^{\pm1}\right]$ such that \[
y^{2}-\lambda^{6}+x+x^{2}=\left(y-\sigma\right)\left(y+\sigma\right)+x^{4}\xi.\]
Note that $\sigma$ and $\xi$, considered as Laurent polynomials
in the variable $\lambda$, are necessarily odd and even respectively.
This identity implies in turn that $V_{\lambda}$ is isomorphic to
the subvariety of ${\rm Spec}\left(\mathbb{C}\left[x,y,z',\lambda^{\pm1}\right]\right)$
defined by the equation $x^{4}z'=\left(y-\sigma\right)\left(y+\sigma\right)$,
where $z'=z-\xi$. The $\mathbb{G}_{a}$-action on $V_{t}$ lift to
the one on $V_{\lambda}$ induced by the locally nilpotent derivation
$x^{4}\partial_{y}+2y\partial_{z'}$. The open subset $V_{\lambda+}=V_{\lambda}\setminus\left\{ x=y+\sigma=0\right\} \simeq{\rm Spec}\left(\mathbb{C}\left[x,\lambda^{\pm1}\right]\left[v\right]\right),$
where $$v=x^{-4}\left(y-\sigma\right)\mid_{V_{\lambda+}}=\left(y-\sigma\right)^{-1}z'\mid_{V_{\lambda+}},$$ is equivariantly isomorphic
to $U_{\lambda}\times\mathbb{G}_{a}$ where $\mathbb{G}_{a}$ acts
on the second factor by translations, and the restriction of the \'etale
morphism ${\rm pr}_{1}:V_{t}\times_{\mathbb{A}_{*}^{1}}\mathbb{A}_{*}^{1}\rightarrow V_{t}$
to $V_{\lambda}\setminus\left\{ x=y+\sigma=0\right\} \simeq U_{\lambda}\times\mathbb{G}_{a}$
yields the expected \'etale trivialization. It follows from this
description that $X^{*}/\mathbb{G}_{a}$ is isomorphic to an algebraic
space obtained as the quotient of disjoint union of $U_{x}=V_{x}/\mathbb{G}_{a}$
and $U_{\lambda}=V_{\lambda+}/\mathbb{G}_{a}$ by a certain \'etale
equivalence relation. Clearly, the only nontrivial part is to check
that $V_{t}/\mathbb{G}_{a}$ is isomorphic to the algebraic space
$\mathfrak{S}_{\lambda}$ of \ref{Par-Quot-Space} above. In view
of I.5.8 in \cite{Kn} it is enough to show that we have a cartesian
square $$ \label{cart} \xymatrix{V_{\lambda +}\times_{V_t} V_{\lambda +} \ar@<0.5ex>[r]^-{{\rm pr}_1} \ar@<-0.5ex>[r]_-{{\rm pr}_2} \ar[d] & V_{\lambda +}=U_{\lambda}\times \mathbb{G}_a \ar[d]_{{\rm pr}_1} \\ R_{\lambda} \ar@<0.5ex>[r]^{s} \ar@<-0.5ex>[r]_{t} & U_{\lambda}.} $$

Letting $g\left(x,\lambda,v\right)=x^{4}v+\sigma\left(x,\lambda\right)\in\mathbb{C}\left[x,\lambda^{\pm1},v\right]$
and $h=\left(x^{4}v+2\sigma\left(x,\lambda\right)\right)v+\xi\left(x,\lambda\right)\in\mathbb{C}\left[x,\lambda^{\pm1},v\right]$,
$V_{\lambda+}\times_{V_{t}}V_{\lambda+}$ is isomorphic to the spectrum
of the ring \[
A=\mathbb{C}\left[x,\lambda^{\pm1},\lambda_{1}^{\pm1},v,v_{1}\right]/\left(g\left(x,\lambda,v\right)-g\left(x,\lambda_{1},v_{1}\right),h\left(x,\lambda,v\right)-h\left(x,\lambda_{1},v_{1}\right),\lambda^{2}-\lambda_{1}^{2}\right)\]
Since $\lambda$ is invertible and $\sigma\left(0,\lambda\right)=\lambda^{3}$,
$x$ and $\sigma$ generate the unit ideal in $\mathbb{C}\left[x,\lambda^{\pm1}\right]$.
It follows that $A$ decomposes as the direct product of the rings
\begin{eqnarray*}
A_{0} & = & \mathbb{C}\left[x,\lambda^{\pm1},v,v_{1}\right]/\left(g\left(x,\lambda,v\right)-g\left(x,\lambda,v_{1}\right),h\left(x,\lambda,v\right)-h\left(x,\lambda,v_{1}\right)\right)\\
 & \simeq & \mathbb{C}\left[x,\lambda^{\pm1},v,v_{1}\right]/\left(x^{4}\left(v-v_{1}\right),x^{4}\left(v^{2}-v_{1}^{2}\right)+2\sigma\left(x,\lambda\right)\left(v-v_{1}\right)\right)\\
 & \simeq & \mathbb{C}\left[x,\lambda^{\pm1},v,v_{1}\right]/\left(x^{4}\left(v-v_{1}\right),2\sigma\left(x,\lambda\right)\left(v-v_{1}\right)\right)\\
 & \simeq & \mathbb{C}\left[x,\lambda^{\pm1}\right]\left[v\right]\end{eqnarray*}
 and \begin{eqnarray*}
A_{1} & = & \mathbb{C}\left[x,\lambda^{\pm1},v,v_{1}\right]/\left(g\left(x,\lambda,v\right)-g\left(x,-\lambda,v_{1}\right),h\left(x,\lambda,v\right)-h\left(x,-\lambda,v_{1}\right)\right)\\
 & \simeq & \mathbb{C}\left[x,\lambda^{\pm1},v,v_{1}\right]/\left(x^{4}\left(v-v_{1}\right)+2\sigma\left(x,\lambda\right),x^{4}\left(v^{2}-v_{1}^{2}\right)+2\sigma\left(x,\lambda\right)\left(v+v_{1}\right)\right)\\
 & \simeq & \mathbb{C}\left[x,\lambda^{\pm1},v,v_{1}\right]/\left(x^{4}\left(v-v_{1}\right)+2\sigma\left(x,\lambda\right)\right)\\
 & \simeq & \mathbb{C}\left[x^{\pm1},\lambda^{\pm1},v,v_{1}\right]/\left(x^{4}\left(v-v_{1}\right)+2\sigma\left(x,\lambda\right)\right)\\
 & \simeq & \mathbb{C}\left[x^{\pm1},\lambda^{\pm1}\right]\left[v\right]\end{eqnarray*}
Thus $V_{\lambda,+}\times_{V_{t}}V_{\lambda+}$ is isomorphic to $R_{\lambda}\times\mathbb{A}^{1}$
and the above diagram is clearly cartesian. This completes the proof
for $X^{*}$.

2) The case of $Y^{*}$. Similarly as for the case of $X^{*}$, $Y^{*}$
is covered by two $\mathbb{G}_{a}$-invariant open subsets \[
W_{x}=Y^{*}\setminus\left\{ x=0\right\} =Y\setminus\left\{ x=0\right\} \qquad\textrm{and}\qquad W_{t}=Y^{*}\setminus\left\{ t=0\right\} =Y\setminus\left\{ t=0\right\} \]
and the morphism \[
U_{x}\times\mathbb{G}_{a}\longrightarrow W_{x},\:\left(x,t,v\right)\mapsto\left(x,x^{4}v,x^{4}\left(1+\alpha x^{2}\right)v^{2}+x^{-4}\left(-t^{3}+x+x^{2}\right),t\right)\]
defines a trivialization of the induced $\mathbb{G}_{a}$-action on
$W_{x}$. To obtain an \'etale trivialization of the $\mathbb{G}_{a}$-action
on $W_{t}$, one checks first that there exists $\zeta\in\mathbb{C}\left[x,\lambda^{\pm 1}\right]$ such that for 
\[\tau=\left(1-\frac{1}{2}\alpha x^{2}\right)\sigma\left(x,\lambda\right)\]
the identity \[\left(1+\alpha x^{2}\right)y^{2}-\lambda^{6}+x+x^{2}=\left(1+\alpha x^{2}\right)\left(y-\tau\right)\left(y+\tau\right)+x^{4}\zeta\left(x,\lambda\right)\]
holds in $\mathbb{C}\left[x,\lambda^{\pm1},y\right]$. Then one checks
in a similar way as above that the morphism \[
U_{\lambda}\times\mathbb{G}_{a}\longrightarrow W_{t},\;\left(x,\lambda,v\right)\mapsto\left(x,x^{4}v+\tau,\left(x^{4}v+2\tau\right)v+\zeta,\lambda^{2}\right)\]
yields an \'etale trivialization, that $W_{t}/\mathbb{G}_{a}\simeq U_{\lambda}/R_{\lambda}=\mathfrak{S}_{\lambda}$
and that $Y^{*}/\mathbb{G}_{a}\simeq\mathfrak{S}$.
\end{proof}
\begin{enavant} From now on, we identify $X^{*}/\mathbb{G}_{a}$
and $Y^{*}/\mathbb{G}_{a}$ with the algebraic space $\mathfrak{S}$
constructed in \ref{Par-Quot-Space} above, and we let $W=X^{*}\times_{\mathfrak{S}}Y^{*}$.
By construction, $W$ is a scheme, equipped with a structure of Zariski
locally trivial $\mathbb{G}_{a}$-bundle over $X^{*}$ and $Y^{*}$
via the first and the second projection respectively. The following
completes the proof.

\end{enavant}
\begin{lem}
We have isomorphisms $X^{*}\times\mathbb{A}^{1}\simeq W\simeq Y^{*}\times\mathbb{A}^{1}$. \end{lem}
\begin{proof}
We will show more precisely that the $\mathbb{G}_{a}$-bundles ${\rm pr}_{1}:W\rightarrow X^{*}$
and ${\rm pr}_{2}:W\rightarrow Y^{*}$ are both trivial. It follows
from the description of the \'etale trivialization given in the proof
of Lemma \ref{lem:Iso-Quot-Spaces} above that the isomorphy class
of the $\mathbb{G}_{a}$-bundle $X^{*}\rightarrow\mathfrak{S}$ in
$H_{\textrm{ét}}^{1}\left(\mathfrak{S},\mathcal{O}_{\mathfrak{S}}\right)$
is represented by the \v{C}ech $1$-cocycle \[
\left(x^{-4}\sigma,2x^{-4}\sigma\right)\in\Gamma\left(U_{x,\lambda},\mathcal{O}_{U_{x,\lambda}}\right)\times\Gamma\left(U_{\lambda\lambda},\mathcal{O}_{U_{\lambda\lambda}}\right)=\mathbb{C}\left[x^{\pm1},\lambda^{\pm1}\right]^{2}\]
with value in $\mathcal{O}_{\mathfrak{S}}$ for the \'etale cover
$p:V\rightarrow\mathfrak{S}$ of $\mathfrak{S}$. Similarly, the isomorphy
class of the $\mathbb{G}_{a}$-bundle $Y^{*}\rightarrow\mathfrak{S}$
is represented by the \v{C}ech $1$-cocycle \[
\left(x^{-4}\tau,2x^{-4}\tau\right)\in\Gamma\left(U_{x,\lambda},\mathcal{O}_{U_{x,\lambda}}\right)\times\Gamma\left(U_{\lambda\lambda},\mathcal{O}_{U_{\lambda\lambda}}\right)=\mathbb{C}\left[x^{\pm1},\lambda^{\pm1}\right]^{2}.\]
This implies in turn that the isomorphy class of the $\mathbb{G}_{a}$-bundle
${\rm pr}_{1}:W\rightarrow X^{*}$ in $H_{\textrm{ét}}^{1}\left(X^{*},\mathcal{O}_{X^{*}}\right)$
is represented by the \v{C}ech $1$-cocyle \[
\alpha=\left(x^{-4}\tau,2x^{-4}\tau\right)\in\Gamma\left(U_{x,\lambda}\times\mathbb{G}_{a},\mathcal{O}_{U_{x,\lambda}\times\mathbb{G}_{a}}\right)\times\Gamma\left(U_{\lambda\lambda}\times\mathbb{G}_{a},\mathcal{O}_{U_{\lambda\lambda}\times\mathbb{G}_{a}}\right)=\left(\mathbb{C}\left[x^{\pm1},\lambda^{\pm1}\right]\left[v\right]\right)^{2}\]
with value in $\mathcal{O}_{X^{*}}$ for \'etale cover given by $U_{x}\times\mathbb{G}_{a}$
and $U_{\lambda}\times\mathbb{G}_{a}$. By definition, ${\rm pr}_{1}:W\rightarrow X^{*}$
is a trivial $\mathbb{G}_{a}$-bundle if and only if $\alpha$ is
coboundary. This is the case if and only if there exists \[
\beta_{x}\in\Gamma\left(U_{x}\times\mathbb{G}_{a},\mathcal{O}_{U_{x}\times\mathbb{G}_{a}}\right)=\mathbb{C}\left[x^{\pm1},t\right]\left[v\right]\quad\textrm{and}\quad\beta_{\lambda}\in\Gamma\left(U_{\lambda}\times\mathbb{G}_{a},\mathcal{O}_{U_{\lambda}\times\mathbb{G}_{a}}\right)=\mathbb{C}\left[x,\lambda^{\pm1}\right]\left[v\right]\]
such that \[
\begin{cases}
x^{-4}\tau & =\beta_{\lambda}\left(x,\lambda,v-x^{-4}\sigma\right)-\beta_{x}\left(x,\lambda^{2},v\right)\\
2x^{-4}\tau & =\beta_{\lambda}\left(x,\lambda,v\right)-\beta_{\lambda}\left(x,-\lambda,v+2x^{-4}\sigma\right)\end{cases}\]
Since \[
\tau\left(x,\lambda\right)=\left(1-\frac{1}{2}\alpha x^{2}\right)\sigma\left(x,\lambda\right),\]
one can choose for instance \[
\beta_{x}\left(x,t,v\right)=-\left(1-\frac{1}{2}\alpha x^{2}\right)v\quad\textrm{and}\quad\beta_{\lambda}\left(x,\lambda,v\right)=-\left(1-\frac{1}{2}\alpha x^{2}\right)v.\]
The fact that  ${\rm pr}_{2}:W\rightarrow Y^{*}$
is also a trivial $\mathbb{G}_{a}$-bundle follows from a similar argument using the identity 
\[ \sigma\left(x,\lambda\right)=\left(1+\frac{1}{2}\alpha x^{2}\right)\tau\left(x,\lambda\right)+{\displaystyle \frac{1}{4}\alpha^2 x^{4}\sigma\left(x,\lambda\right)}.\] 
\end{proof}

\end{document}